\documentclass{article}

\usepackage{amsmath,amssymb,theorem}
\usepackage{epsfig}
\usepackage{graphicx,color}
\usepackage{makeidx}
\usepackage{multicol}
\usepackage{a4wide}

\newcommand{\N}{\mathbb{N}}
\newcommand{\RR}{\mathbb{R}}                                  
\newcommand{\R}{\mathbb{R}}                                  

\newcommand{\intt}{\mathrm {int}\,}
\newcommand{\cl}{\mathrm {cl}\,}

\newcommand{\conv}{\mathrm{conv}\,}
\newcommand{\midset}{\;|\;}
\newcommand{\exppi}{\pi_{\mathrm{exp}}}

\DeclareMathOperator{\ecc}{ecc}
\DeclareMathOperator{\dist}{dist}

\newcounter{enumctr}
\newenvironment{enum}{\begin{list}{(\roman{enumctr})}{\usecounter{enumctr}}}{\end{list}}

\newtheorem{theorem}{Theorem}[section]

\newtheorem{definition}[theorem]{Definition}
\newtheorem{lemma}[theorem]{Lemma}
\newtheorem{corollary}[theorem]{Corollary}
\newtheorem{proposition}[theorem]{Proposition}

{ \theorembodyfont{\normalfont} 

\newtheorem{remark}[theorem]{Remark}
}

\begin{document}

\title{Extremal norms for positive linear inclusions.}

\author{%
 Oliver Mason\footnotemark[2]
      \thanks{Hamilton Institute,
      National University of Ireland, Maynooth, Ireland, {\tt
       oliver.mason@nuim.ie}}, 
     \and Fabian Wirth\thanks{Institut f\"{u}r
      Mathematik, Universit\"{a}t W\"{u}rzburg, W\"{u}rzburg, Germany,
      {\tt wirth@mathematik.uni-wuerzburg.de}
F. Wirth is supported by the DFG under grant Wi1458/10.}}

\maketitle

\begin{abstract}
  For finite-dimensional linear semigroups which leave a proper cone invariant
  it is shown that irreducibility with respect to the cone implies the
  existence of an extremal norm. In case the cone is simplicial a similar
  statement applies to absolute norms. The semigroups under consideration may
  be generated by discrete-time systems, continuous-time systems or
  continuous-time systems with jumps.  The existence of extremal norms is used
  to extend results on the Lipschitz continuity of the joint spectral radius
  beyond the known case of semigroups that are irreducible in the
  representation theory interpretation of the word.
\end{abstract}


\smallskip


\section{Introduction}

In this paper we investigate the exponential growth rate of linear
semigroups which leave a proper cone $K$ invariant. In the case of matrix
products this exponential growth rate is commonly known as the joint
spectral radius of a set of matrices and we will use this name for matrix
semigroups in general.  For matrix semigroups generated by discrete-time
systems, continuous time systems or switched systems with jumps it is
shown that an irreducibility condition of the semigroup with respect to
the cone $K$ guarantees the existence of an extremal norm. In addition
this norm can be chosen to be monotone with respect to the cone $K$. This
has been obtained previously for positive semigroups with respect to the
cone $\R^n_+$ in \cite{ABMW12} and for general cones in the discrete time
case in \cite{guglielmi2013exact}. Here we follow the basic idea of
\cite{ABMW12}, but have to adjust several arguments to be able to deal
with general proper cones. The approach is distinct from the ideas
presented in \cite{guglielmi2013exact}.

 The
results are used to extend regularity results for the joint spectral
radius, in that the joint spectral radius is irreducible in a neighbourhood
of $K$-irreducible semigroups.

The joint spectral radius of sets of matrices for discrete or continuous
linear inclusions and associated extremal norms have been studied in
\cite{Gurv95,Kozy90,Wirt02,wirth2005structure,guglielmi2008algorithm,maesumi2008optimal,morris2010criteria,dai2011extremal},
see also the survey \cite{Sho07}.  The joint spectral radius plays a key role
in characterising growth rates of solutions of inclusions and hence in their
stability analysis \cite{Jungers}.  Motivated by the practical importance of
systems whose state variables are constrained to remain non-negative (given
non-negative initial conditions), there has been significant interest recently
in studying positive inclusions and positive switched systems
\cite{FainMarChig, GurShoMas,jungers2012asymptotic,Pepe12}. Also in several recent
proposals for efficient computation of the joint spectral radius, positivity
has played an important role, see
\cite{blondel2005accuracy,guglielmi2013exact} and the lifting techniques
discussed for instance in \cite{jungers2012lifted}.  In
\cite{protasov2010joint} positivity with respect to arbitrary positive cones
has been considered. Also it has been noted e.g. in \cite{morris2010criteria}
that positivity properties yield criteria for the uniqueness (up to scaling)
of Barabanov norms, a concept we describe below.  Note that in general it is
hard to decide, whether a given set of matrices share an invariant cone,
\cite{Prot10}. In some cases, however, this is clear in the way that the
matrices are defined. In particular, for the positive orthant $\R^n_+$ or the
cone or positive semidefinite matrices.

The defining property of positive inclusions, the positivity constraint on
their dynamics, has motivated a number of particular questions in their
stability analysis.  One such example is in the study of \emph{copositive
  Lyapunov functions} and linear copositive Lyapunov functions in
particular.  Necessary and sufficient conditions for the existence of a
common linear copositive Lyapunov function in the case of a switched
positive system with 2 modes were given in \cite{MasSho07}; these were then
extended to the general case in \cite{knorn2009linear}.  Copositivity has
subsequently been studied for more general cones in \cite{Bundfuss}, while
its application to the stability and stabilisation of switched positive
systems has been thoroughly investigated in
\cite{fornasini2010stabilizability}, \cite{fornasini2010linear} and
\cite{fornasini2012stability}.

In \cite{mason2009stability}, the question of D-stability for positive inclusions 
was considered.  Separate sufficient and necessary conditions for a positive
linear inclusion to be D-stable were described.  These conditions are 
intimately connected to the existence of common linear copositive Lyapunov 
functions for a related inclusion. This theme was developed in \cite{BokMasWir2010} 
where a single necessary and sufficient condition was described for a system 
whose constituent systems are described by irreducible matrices.  The matrix 
theoretic notion of irreducibility shall play a key role again in the current paper.  

It was shown in \cite{GurShoMas} that the stability of a 2-dimensional
positive inclusion is equivalent to the stability of the convex hull of
its associated matrices.  Unfortunately this result fails to be true in
general and a specific counterexample for 3-dimensional systems has been
described in \cite{FainMarChig}.  In studying general linear inclusions,
the concepts of extremal and Barabanov norms play an important role as
such norms can be used to explicitly characterise the growth rate of an
inclusion.  Barabanov norms for positive linear inclusions have been
recently considered in \cite{Teichner201295}; an explicit closed-form
expression for a Barabanov norm for 2-dimensional discrete positive
inclusions is derived. 

In this paper, we consider extremal norms for positive inclusions in both
the discrete and continuous case.  We show that for positive inclusions, a
matrix-theoretic notion of irreducibility is sufficient for the existence
of an extremal norm.  This is novel as the concept of irreducibility we
consider is weaker than that used to establish the existence of extremal
norms in the work of Barabanov{, \cite{Bara88}} and others.
Our result also relates the existence of an extremal norm to a property of
the convex hull of the matrices associated with the inclusion.

We use the above results to extend regularity results for the joint
spectral radius \cite{Victor201012, Wirt02}. In the case of positive
systems irreducibility in the sense of nonnegative matrices is sufficient
for local Lipschitz properties of the joint spectral radius.  We emphasise
that this notion of irreducibility is distinct from that utilised for
general inclusions.  Our results show that an interesting observation of
\cite{kozyakin2007structure} for the case of nonnegative $2\times 2$
matrices is a consequence of a general property of sets of nonnegative
matrices and monotone norms.

The paper is organised as follows: In the ensuing
Section~\ref{sec:prelims} we recall several known results about proper
cones and the properties of matrices that leave such cones invariant. In
Section~\ref{sec:linincl} we describe three principles by which linear
semigroups can be generated and recall a common definition of the joint
spectral radius. In Section~\ref{sec:exnorms} we recall the definition of
extremal norms. As the first main result it is shown that a
$K$-irreducible positive semigroup admits a monotone extremal norm. The
results of Section~\ref{sec:exnorms} are used in
Section~\ref{sec:regularity} to obtain Lipschitz continuity results for
the joint spectral radius. These results extend previous results in
\cite{Wirt02} which used a notion of irreducibility from representation
theory to obtain Lipschitz continuity properties.
\subsection{Preliminaries}
\label{sec:prelims}

  Throughout the paper, $\RR$ and $\RR^n$ denote the field of real numbers
  and the vector space of all $n$-tuples of real numbers, respectively.
  For $x \in \RR^n$ and $i = 1, \ldots , n$, $x_i$ denotes the $i$th
  coordinate of $x$. Similarly, $\RR^{n \times n}$ denotes the space of $n
  \times n$ matrices with real entries and for $A \in \RR^{n \times n}$,
  $A_{ij}$ denotes the $(i, j)$th entry of $A$.  The convex hull of a set
  $C\subset \R^n$ is denoted by $\conv C$.

  We now recall standard concepts regarding proper cones, as discussed in
  \cite{vandergraft1968spectral,schneider1970cross,tam2001cone}. As usual,
  a cone $K\subset \R^n$ is a nonempty set satisfying $rK \subset K$ for
  all real $r>0$. We will always consider a proper cone $K$, that is a
  cone, which is (i) convex, so that $x+y\in K$ for all $x,y\in K$, (ii)
  pointed, i.e. $K\cap -K = \{ 0 \} $, (iii) closed and (iv) full,
  i.e. the interior of $K$ is nonempty. The interior of a set $C\subset
  \R^n$ is denoted by $\intt C$, the closure by $\cl C$ and its boundary
  by $\partial C$. The ball of radius $\varepsilon>0$ around $0\in \R^n$
  is denoted by $B(0,\varepsilon)$. 

  Given a proper cone $K$ we consider the partial order induced by $K$ on
  $\RR^n$. For vectors $x, y \in \RR^n$, we write: $x \geq_K y$ if
  $x-y\in K$; $x >_K y$ if $x \geq_K y$ and $x \neq y$; $x \gg_K y$ if
  $x-y\in \intt K$. A base $B$ of a cone $K$ is a set with the
  properties that 
  \begin{equation*}
      0 \notin \cl B \quad \text{ and } \quad K = \R_+ B = \{ rx \midset
      x\in B, r\geq 0 \} \,. 
  \end{equation*}
  We will always consider compact bases, that are given as the
  intersection of a hyperspace $X$ in $\R^n$ with $K$. It is known that a
  cone has such a base, if and only if it is pointed.

  It will be useful to study norms, which are adapted to the nonnegative
  setting.  A norm $\|\cdot\|$ on $\RR^n$ is called {\em monotone} with
  respect to the ordering induced by $K$ if 
  \begin{equation}
      \label{eq:monnorm}
x\geq_K y\geq_K 0 \quad \Rightarrow \quad 
  \|x\|\geq \|y\|\,.
  \end{equation}
 For any proper cone and the corresponding order, 
  a monotone norm on $\R^n$ exists, see \cite[p.~38]{KrasLifs89}.

  A matrix $A\in \R^{n \times n}$ is called nonnegative, or cone
  preserving, with respect to a proper cone $K$ if $AK\subset K$. The set
  of cone-preserving linear maps is denoted by $\pi(K)$.  We say that $A$
  is $K$-positive if $A\left(K \setminus \{ 0 \} \right) \subset \intt K$.

  Following \cite{stern1991exponential} we say that $A$ is exponentially
  $K$-nonnegative, if $e^{At} \in \pi(K)$, i.e., $e^{At}K \subset K$, for
  all $t\geq0$. This property is also known as {\em cross-positivity} of
  $A$ with respect to the cone $K$, see
  \cite{schneider1970cross,gritzmann1995cross}. We denote the set of
  exponentially $K$-nonnegative matrices by $\exppi(K)$. Also $A$ is
  exponentially $K$-positive, if $e^{At}$ is $K$-positive, for all
  $t>0$.

  For the concept of irreducibility of a cone-preserving map $A\in
  \pi(K)$, recall that a face $F$ of a proper cone $K$ is a cone contained
  in $K$, which also has the property, that 
  \begin{equation}
      \label{eq:facedef}
      x\in F \text{ and } x\geq_K y \geq_K
  0 \quad \Longrightarrow \quad  y\in F \,.
  \end{equation}
   The faces $\{ 0 \}$ and $K$ are the trivial cases of
  faces of $K$.  Now $A\in \pi(K)$ is called $K$-reducible, if $AF\subset
  F$ for some nontrivial face $F$ of $K$ and $K$-irreducible, if it is not
  $K$-reducible. We will need the following characterisation of
  $K$-irreducibility, \cite[Theorem~4.1 and
  Lemma~4.2]{vandergraft1968spectral}: an  $A\in \pi(K)$ is
  $K$-irreducible, if and only if one of the following equivalent
  conditions is satisfied:
  \begin{enum}
    \item[(IR1)] $A$ does not have an eigenvector on $\partial K$,
    \item[(IR2)] $(I+A)^{n-1} (K\setminus \{ 0 \}) \subset \intt K$.
    \item[(IR3)] For every compact base $B$ of $K$ we have
    $(I+A)^{n-1} B \subset \intt K$.
  \item[(IR4)] For every compact base $B$ of $K$ and all choices of
    positive coefficients $\alpha_i>0, i=1,\ldots,n-1$ we have for
    corresponding linear combination of the powers of $A$ that
      \begin{equation*}
          \left( \sum_{i=1}^{n-1} \alpha_i A^i\right) B \subset \intt K\,. 
      \end{equation*}
  \end{enum}

  The equivalence of $K$-irreducibility and (IR1) is proved in
  \cite{vandergraft1968spectral}, as well as the fact that
  $K$-irreducibility implies (IR2). The remaining implications are trivial.

  \begin{remark}
\label{rem:irred}
      Note that from condition (IR2) or (IR3) we see that
      $K$-irreducibility is an open property in $\pi(K)$ in the topology
      induced by the norm topology on $\R^{n \times n}$. Also the set of
      $K$-positive matrices is open in $\R^{n \times n}$.
  \end{remark}

  An exponentially $K$-nonnegative matrix has the property that $e^{A t}$
  is $K$-irreducible for all $t$ with the possible exception of a discrete
  set if and only if $A$ has no eigenvector in $\partial K$, see
  \cite[Lemma~8]{schneider1970cross}. We require a 
  slightly different view
  of this statement as follows. Again let $K$ be a proper cone. For any
  face $F$ of $K$ we denote by $H_F$ the span of $F$, that is the smallest
  linear space containing $F$. Recall, that for any vector $x\in \partial
  K$, $x\neq 0$, there exists a face $F$ of $K$, such that $x \in
  \intt_{H_F} F$, where $\intt_{H_F}F$ denotes the interior of $F$ relative to
  the subspace $H_F$, \cite[Lemma~2.1]{vandergraft1968spectral}.

  \begin{lemma}
\label{lem:expnonneg}
      Let $A\in \R^{n \times n}$ be exponentially $K$-nonnegative. The
      following are equivalent:
      \begin{enum}
        \item $A$ has no eigenvector in $\partial K$,
        \item $e^{A t}$ is $K$-irreducible for all $t\geq 0$ with the possible
          exception of a discrete set,
        \item there exists a $t>0$ such that $e^{A t}$ is $K$-irreducible.
        \item there does not exist a nontrivial face $F$ of $K$ such that
          $H_F$ is $A$-invariant.
      \end{enum}
  \end{lemma}

  \begin{proof}
      The equivalence of (i) and (ii) is proved in
      \cite[Lemma~8]{schneider1970cross}. The implication (ii)
      $\Rightarrow$ (iii) is obvious. If (iii) holds, then $e^{At}$ does
      not have an eigenvector in $\partial K$ by (IR1). As the
      eigenvectors of $A$ are also eigenvectors of $e^{At}$ we obtain (i).
      If (iv) does not hold, then there exists a nontrivial face $F$ of
      $K$ such that $A_{|H_F}$ is exponentially $F$-nonnegative. By
      \cite[Theorem~6]{schneider1970cross} this implies that $A_{|H_F}$ has
      an eigenvector in $F$. This is also
      an eigenvector of $A$ so that (i) is false. 

      Finally, let $Ax=\lambda x$ for $x\in\partial K\setminus \{ 0 \}$.
      Now let $F$ with span $H_F$ such that $x\in \intt_{H_F} F$. For any
      $t>0$ we have that $e^{At} x = e^{\lambda t} x$ and for $0\neq y \in
      F$ we may choose an $\alpha >0$ such that $\alpha y \ll_F x$, as $x
      \in \intt_{H_F} F$. This implies $\alpha y \leq_K x$ and so $0
      \leq_K e^{At} y \leq_K e^{At} x = \alpha e^{\lambda t} x$. By the
      defining property of a face this implies that $e^{At} y \in F$. As
      $t>0$ was arbitrary we see that $e^{At}F \subset F$ for all $t>0$
      from which it follows that $e^{At}H_F \subset H_F$ for all $t>0$ and
      so $A H_F \subset H_F$.  This concludes the proof.
  \end{proof}

  We note that in the proof of $\neg$ (i) $\Rightarrow \neg$ (iv), we have
  followed ideas already used in \cite{vandergraft1968spectral}. With a
  slight but common abuse of terminology, we call an exponentially
  $K$-nonnegative matrix $A$ irreducible, if $A$ satisfies one of the
  equivalent conditions of Lemma~\ref{lem:expnonneg}.

  \begin{remark}
      Concerning Lemma~\ref{lem:expnonneg}\,(ii), note that the
      corresponding statement in \cite{schneider1970cross} always speaks
      of an at most countable exceptional set of times $t$ at which
      $e^{At}$ is not irreducible. As the proof in
      \cite{schneider1970cross} uses analyticity arguments the formulation
      we use here is actually the statement proved in the original paper.
\hfill $\Box$
  \end{remark}

  We will also need the following two observations about cones and spectra
  of matrices.

  \begin{lemma}
      \label{lem:base}
      Let $K\subset \R^n$ be a proper cone and let $B$ be a compact base
      of $K$.  For every compact $C\subset \intt K$ there is a $\delta>0$
      such that
\begin{equation}
    \label{eq:lembase}
    \delta x \ll_K y\,, \quad  \forall\, x \in B, y\in C\,.
\end{equation}
  \end{lemma}

\begin{proof}
    Let $y\in C$. By assumption $y\gg_K 0$, so $0 \in \intt (y - K)$. Let
    $\varepsilon>0$ be small enough so that $B(0,\varepsilon) \subset
    \intt (y - K)$. As $B$ is compact there exists a $\delta_y>0$ such
    that $\delta_y B \subset B(0,\varepsilon/2)$. Then if $\|y-z\|<
    \varepsilon/2$ we have
    \begin{equation*}
        (y-z) + \delta_y B \subset B(0,\varepsilon) \subset \intt (y - K)
    \end{equation*}
    and so $\delta_y B \subset \intt (z - K)$. Thus for every $y\in C$
    there exists a neighbourhood $U_y$ and a $\delta_y>0$ such that
    $\delta_y x \ll_K z$ for all $x \in B, z \in U_y$. Choose a finite
    subcover $U_{y_1}, \ldots, U_{y_\ell}$ of the cover $\{ U_y \midset
    y\in C \} $. Then $\delta := \min \{ \delta_{y_1}, \ldots, \delta_{y_\ell} \}
    >0$ satisfies \eqref{eq:lembase}.
\end{proof}

  \begin{lemma}
      \label{lem:specrad}
Let $K\subset \R^n$ be a proper cone and $A\in \pi(K)$. If there exist $x\in K, c>0$ such that
\begin{equation*}
    A x \gg_K c x\,,
\end{equation*}
then the spectral radius $r(A) > c$.
  \end{lemma}

  \begin{proof}
   This follows readily from Corollary 1.3.34 of \cite{BP87}.
  \end{proof}

  A cone in $\R^n$ is called {\em simplicial} if it is given as the conical
  hull of exactly $n$ linearly independent vectors. A particular example
  is the positive orthant $\RR^n_+ := \{x \in \RR^n: x_i \geq 0, 1 \leq i
  \leq n\}$.  Simplicial cones are precisely the cones which induce a
  lattice structure, i.e. an ordering which admits the definition of
  maximum and minimum. We define $\max_K \{ x, y \}$ as the unique
  $z\in \R^n$ with the property that $z\geq_K x$ and $z\geq_K y$ and
  \begin{equation*}
       w \geq_K x \text{ and } w \geq_K y \quad \Longrightarrow \quad w \geq_K z\,.
  \end{equation*}
  The absolute value $|x|_K$ of a vector $x\in \RR^n$ with respect to a
  simplicial cone $K$ is then given by
  \begin{equation*}
      |x|_K := \max{} _K \{ x, -x \}\,.
  \end{equation*}
  Again in the case that $K= \RR^n_+$ we have that $|x|=|x|_{\RR^n_+}$ is defined by
  $|x|_i := |x_i|, i=1,\ldots,n$.

  If the cone is simplicial it is usual to sharpen the definition of a
  monotone norm to the requirement that
  \begin{equation*}
      |x|_K \geq_K |y|_K \quad \Rightarrow \quad \|x\| \geq \|y\|\,.
  \end{equation*}
  This is equivalent to the requirement that $\|x\| = \| \, |x|_K\, \|$ for
  all $x\in \RR^n$, see \cite[Theorem~2]{bauer1961absolute},
  \cite[Theorem~5.5.10]{HornJohn}.  Norms with the latter property are
  called {\em absolute}. Since there is the potential of misinterpretation
  of the term monotone, we will use monotone to denote the property
  defined in \eqref{eq:monnorm} and we will always speak of absolute
  norms, when we have a simplicial cone.

\section{Linear Inclusions and the Joint Spectral Radius}
\label{sec:linincl}

We now discuss semigroups of matrices ${\cal S}\subset \R^{n \times n}$
that have an associated concept of time and thus a growth rate. There are
several ways of introducing such semigroups and we discuss three of
them. The first corresponds to discrete time switched systems, the second
to continuous time switched systems and the third one encompasses switched
differential algebraic systems, \cite{Sho07,trenn2012linear}.

\smallskip

{\parindent0pt \bf A. The Discrete Time Case}

For a compact set ${\cal M}\subset \R^{n \times n}$ we consider the linear
inclusion
\begin{equation}
    \label{eq:lininpos}
    x(t+1) \in \{ Mx(t) \midset M \in {\cal M}\} \,, t\in \N\,.
\end{equation}

Solutions of \eqref{eq:lininpos} corresponding to the initial value $x_0$
are given by sequences $\{ x(t) \}_{t\in\N}$ where for each $t\in\N$ there
exists an $A(t)\in {\cal M}$ such that $x(t+1)=A(t)x(t)$. The evolution
operators generated by ${\cal M}$ are therefore the sets
\begin{equation*}
    {\cal S}_t:= \{ A(t-1) \dots A(0) \midset A(s) \in {\cal M}, s=0,\ldots,t-1 \} \,,
\end{equation*}
and the associated matrix semigroup is ${\cal S}:= \bigcup_{t\in\N} {\cal
  S}_t$, where we set ${\cal S}_0 :=\{I\}$. If we want to emphasise that
${\cal S}$ is generated by ${\cal M}$ in the discrete time setting, we
write ${\cal S} = {\cal S}({\cal M}, \N)$.

\smallskip

{\parindent0pt \bf B. The Continuous Time Case}

In the continuous time setting we define linear inclusions as
follows. Given a compact set of matrices ${\cal M}$ we consider a
linear inclusion of the form
\begin{equation}
    \label{eq:lininmetz}
    \dot{x} \in \{ Mx \midset M \in {\cal M}\} \,.
\end{equation}

The evolution operators defined by \eqref{eq:lininmetz} are given by
solutions to the differential equation
\begin{equation*}
    \dot \Phi_\sigma(t) = A(t) \Phi_\sigma(t)\,,\quad \Phi_\sigma(0) =I\,, 
\end{equation*}
where $\sigma:=A:\RR_+ \to {\cal M}$ is measurable. It is
possible to consider only piecewise continuous functions $A$ with locally
finitely many discontinuities; this neither changes the notions of
positivity nor of stability discussed below. The map $\sigma$ is called
the switching signal defining the differential equation and the notation
$\Phi_\sigma$ is a reminder that it is defined via a particular
switching signal. In this case the set of time $t$ evolution operators is given by
\begin{equation*}
    {\cal S}_t := \{ \Phi_\sigma(t) \midset \sigma: [0,t] \to {\cal M} \text{ measurable} \} 
\end{equation*}
and again ${\cal S}:= \bigcup_{t\in\RR_+} {\cal S}_t$, where we set ${\cal
  S}_0 :=\{I\}$. If we want to emphasise that ${\cal S}$ is generated by
${\cal M}$ in the continuous time setting, we write ${\cal S} = {\cal
  S}({\cal M}, \R_+)$.

\smallskip

{\parindent0pt \bf C. Switched Linear Systems with Jumps}

We now discuss a class of impulsive systems which encompasses the case of
switched DAEs, as shown in \cite{trenn2012linear}. Consider a compact set
${\cal M} \subset \R^{n \times n} \times \R^{n \times n}$, where each pair
$(A,\Pi) \in {\cal M}$ has the property that
\begin{equation*}
    \Pi^2 = \Pi \,,\quad A\Pi = \Pi A \,.
\end{equation*}

For a piecewise constant and right-continuous switching signal
$\sigma:\R_+ \to {\cal M}$, $t\mapsto (A_{\sigma(t)}, \Pi_{\sigma(t)})$,
with discontinuities $0=t_0 < t_1 < t_2 < \ldots < t_k \to \infty$ we
define the corresponding evolution operator for $t\in [t_k,t_k+1)$ by
\begin{equation}
\label{eq:DAEevdef}
    \Phi_\sigma(t,0) = e^{A
_{\sigma(t)} (t-t_{k-1})} \Pi_{\sigma(t_{k-1})} \ldots 
e^{A
_{\sigma(t_1)} (t_2-t_1)} \Pi_{\sigma(t_1)}e^{A
_{\sigma(t_0)} (t_1-t_0)} \Pi_{\sigma(t_0)}\,.
\end{equation}
Again, for every $t\geq 0$, ${\cal S}_t$ is defined as the set of all
possible evolution operators that can be defined by \eqref{eq:DAEevdef}
and ${\cal S}:= \bigcup_{t\in\RR_+} {\cal S}_t$. It is shown in
\cite[Lemma~6]{trenn2012linear} that this defines a semigroup. Also it is
shown in Proposition~10 of that reference that the sets ${\cal S}_t$ are
bounded as subsets of $\R^{n \times n}$ if and only if the set of
projections
\begin{equation*}
   {\cal M}_\Pi:=  \{ \Pi \;|\; \exists\ A \in \R^{n \times n} \text{ such that } (A,\Pi) \in {\cal M} \} 
\end{equation*}
is product bounded. The set of projections is product bounded if the
discrete semigroup ${\cal S}({\cal M}_\Pi, \N)$ is bounded.

\smallskip

{\parindent0pt \bf Growth Rates}

Given a semigroup defined in one of the ways we have discussed so far,
we define the joint spectral radius of ${\cal S}$ by setting
\begin{equation}
    \label{eq:rhodef}
    \rho({\cal S}) := \lim_{t\to \infty} \sup \{ \|S\| \midset S\in {\cal S}_t \}^{1/t}\,,
\end{equation}
where we will suppress the fact that depending on the situation at hand
$t\in \N$ or $t\in \R_+$.

It is well known that the limit exists, is independent of the norm
considered, and characterises the maximal { and uniform} exponential
growth of solutions to \eqref{eq:lininpos}, resp. \eqref{eq:lininmetz} or
\eqref{eq:DAEevdef} \cite{Bara88,Sho07,Jungers,trenn2012linear}. We will
need the following property of the joint spectral radius, which is
independent of the particular way in which the semigroup is defined:
\begin{equation}
    \label{eq:gsr}
    \rho({\cal S}) := \limsup_{t\to \infty} \sup \{ r(S) \midset S\in {\cal S}_t \}^{1/t}\,.
\end{equation}
Also in the discrete and the continuous time case we have by \cite{Bara88}
the convexity relation
\begin{equation}
    \label{eq:jsrconv}
    \rho({\cal S}({\cal M})) = \rho( {\cal S}(\conv {\cal M}))\,.
\end{equation}
An immediate consequence of \eqref{eq:gsr} and \eqref{eq:jsrconv} is the
property that
\begin{equation}
    \label{eq:convspec}
    r(A) \leq \rho({\cal S}({\cal M})) \,,\quad \forall A\in \conv {\cal M}\,.
\end{equation}
This property is proved using extremal norms in \cite[Part I]{Bara88} and an
alternative argument for this relation is provided in
\cite[Lemma~1]{blondel2005computationally}.

In the following we concentrate on $K$-positive switched systems for a
regular cone $K$.

It is clear that in the discrete time case ${\cal S}({\cal M},\N) \subset
\pi(K)$, if and only if ${\cal M}\subset \pi(K)$, whereas in the
continuous time case ${\cal S}({\cal M}, \R_+) \subset \pi(K)$ if and only
if ${\cal M}\subset \exppi(K)$. From these two relations it is easy to
see, that in the case of switched systems with jumps, we have that
$K$-positivity is equivalent to the to requirements (i) ${\cal M}_\Pi
\subset \pi(K)$ and (ii) $A\Pi \in \exppi(K)$ for all $(A,\Pi)\in {\cal M}$.

For positive systems irreducibility of matrices plays a decisive role.
\begin{definition}{{\bf (Irreducibility of Inclusions)}}
    \label{d:irred}
    Let $K\subset \R^n$ be a proper cone. A semigroup ${\cal S}\subset
    \pi(K)$ is called $K$-irreducible, if there exists a $t>0$ such that
    $\conv {\cal S}_t$ contains a $K$-irreducible element.
\end{definition}

While the previous definition has the advantage of being independent of
the particular definition of the semigroup, it is instructive to point out
what the definition amounts to in the different cases.  To this end the
following observation is of interest. In the continuous-time case we
provide an argument for matrices of the form $A+\lambda I, A\in 
\pi(K), \lambda \in \R$. Denoting $\Lambda := \{ \lambda I \midset
\lambda\in \R \}$, it is known that $\pi(K)+\Lambda \subset \exppi(K)$. For
general cones the two sets are distinct, but  for polyhedral cones
equality holds,
\cite{schneider1970cross,gritzmann1995cross}.

\begin{proposition}
    \label{p:irred}
    Let $K\subset \R^n$ be a proper cone.
    \begin{enum}
      \item If $\emptyset\neq {\cal M} \subset \pi(K)$, then
        $\conv {\cal M}$ contains a $K$-irreducible element if and only if
        for every nontrivial face $F$ of $K$ there exists an $A\in {\cal
          M}$ such that $AF \not \subset F$.
      \item If $\emptyset\neq {\cal M} \subset \pi(K)+\Lambda$ is bounded,
        then $\conv {\cal M}$ contains a $K$-irreducible element if and
        only if for every nontrivial face $F$ of $K$ there exists an $A\in
        {\cal M}$ such that $A\, \mathrm{span}\ F \not \subset
        \mathrm{span}\ F$.
    \end{enum}
 \end{proposition}

\begin{proof}
    (i) As $A\in \pi(K)$ is $K$-irreducible if and only if $rA$ is for
    every $r>0$, we may replace every nonzero $A\in{\cal M}$ by
    $A/\|A\|$ and obtain a bounded set ${\cal M}$. It is thus sufficient
    to prove the claim for bounded sets ${\cal M}\subset \pi(K)$.

    If there exists a nontrivial face $F$ of $K$ such that $AF\subset F$
    for all $A\in {\cal M}$, then the same is clearly true for all $A\in
    \conv {\cal M}$ and so $\conv {\cal M}$ does not contain a
    $K$-irreducible element.

    Conversely, assume that for every nontrivial face $F$ of $K$ there
    exists an $A\in {\cal M}$ such that $AF \not \subset F$. By
    Remark~\ref{rem:irred} it is sufficient to show that $\cl \conv {\cal
      M}$ contains a $K$-irreducible element. Let ${\cal Q} = \{ A_1, A_2,
    \ldots \}$ be a dense sequence lying in ${\cal M}$ and choose a
    sequence $\{ \varepsilon_k \}_{k\in\N}$ such that
    \begin{equation*}
        \sum_{k=1}^\infty \varepsilon_k =1 \,,\quad\text{ and }\quad \varepsilon_k >0\,, \ \  \forall k\, \in \N\,.
    \end{equation*}
    By construction $\overline{A}:=\sum_{k=1}^\infty \varepsilon_k A_k \in
    \cl \conv {\cal M}$. For every nontrivial face $F$ of $K$ we may
    choose an $A\in {\cal M}$ such that $AF \not \subset F$ and as this is
    an open property there exists an index $j$ such that $A_jF \not
    \subset F$.  Now for any $B_1,B_2 \in \pi(K)$ we have as a consequence
    of \eqref{eq:facedef} that $\left(B_1+B_2\right)F \subset F$ implies
    $B_iF \subset F, i=1,2$. With this argument it follows that
    $\overline{A} F \not \subset F$ and as $F$ was arbitrary this shows
    that $\overline{A}$ is $K$-irreducible.

    (ii) This follows from (i), by considering ${\cal M} + rI \subset
    \pi(K)$ for some $r>0$ large enough.
\end{proof}

\begin{proposition}
\label{lem:irredchar}
Let $K\subset \R^n$ be a proper cone.
\begin{enum}
  \item Let ${\cal S}={\cal S}({\cal M}, \N)\subset \pi(K)$ be generated
    via the discrete inclusion \eqref{eq:lininpos}. Then ${\cal S}$ is
    $K$-irreducible if and only if $\conv {\cal M}$ contains a $K$-irreducible
    element.
  \item Let ${\cal S}={\cal S}({\cal M}, \R_+)\subset \pi(K)$ be generated
    via the continuous inclusion \eqref{eq:lininmetz}.  Then ${\cal S}$ is
    $K$-irreducible if and only if for every nontrivial face $F$ of $K$
    there exists an $A\in {\cal M}$ such that $A\, \mathrm{span}\ F \not
    \subset \mathrm{span}\ F$.

    In case that ${\cal M} \subset \pi(K) + \Lambda$ the latter statement
    is equivalent to the existence of an
    irreducible, exponentially nonnegative $A\in\conv {\cal M}$.
\end{enum}
\end{proposition}

\begin{proof}
    (i) As ${\cal S}_1={\cal M}$, it is clear that if $\conv {\cal M}$
    contains a $K$-irreducible element, then ${\cal S}$ is
    irreducible. Conversely, if $\conv {\cal M}$ does not contain a
    $K$-irreducible element, then by Proposition~\ref{p:irred}\,(i) there
    exists a nontrivial face $F$ of $K$ such that $AF\subset F$ for all
    $A\in {\cal M}$. It follows that for any $t\in \N$ we have $SF \subset
    F$ for all $S \in {\cal S}_t$. Then $SF \subset F$ for all $S \in
    \conv {\cal S}_t$ and as $t$ is arbitrary, ${\cal S}$ is
    $K$-reducible.

    (ii) If there exists a nontrivial face $F$ of $K$ such that for all
    $A\in {\cal M}$ we have $AH_F \subset H_F$, then for all $A\in {\cal
      M}$ and all $t\geq0$ we have $e^{At} F \subset F$. By classical
    relaxation arguments, we have that products of the form
    \begin{equation*}
        e^{A_k t_k} e^{A_{k-1}t_{k-1}} \ldots e^{A_1 t_1} \,,\quad
        A_j \in {\cal M}, t_j>0, j=1,\ldots, k, \sum_{j=1}^k t_j = t
    \end{equation*}
    lie dense in ${\cal S}_t({\cal M})$. This shows that $SF \subset F$
    for all $S\in \conv {\cal S}_t$ and as $t>0$ is arbitrary it follows
    that ${\cal S}$ is reducible.

    Conversely, assume that for every nontrivial face $F$ of $K$ there
    exists an $A\in {\cal M}$ such that $AH_F \not \subset H_F$. Choosing
    a dense sequence ${\cal Q} = \{ A_1, A_2, \ldots \}$ lying in ${\cal
      M}$ we have as before that for every nontrivial face $F$ of $K$
    there exists an $A_j\in {\cal M}$ such that $A_jH_F \not \subset
    H_F$. For every index $j$ consider the set $T_j \subset \R_+$ of times
    $t$ for which the invariant subspaces of $e^{A_j t}$ coincide with
    those of $A_j$. It is well known that $T_j$ is the complement of a set
    of Lebesgue measure $0$. It follows that there exists a $\bar{t} \in
    \bigcap_{j\in\N}T_j$. Thus for all nontrivial faces $F$ of $K$ there
    exists an index $j$ such that $e^{A_j \bar{t}}F \not \subset F$.  It
    then follows from Proposition~\ref{p:irred}\,(i) that $\conv {\cal
      S}_{\bar{t}}$ contains an irreducible element.

The final statement follows from Proposition~\ref{p:irred}\,(ii).

\end{proof}

\section{Extremal Norms}
\label{sec:exnorms}

In the analysis of linear inclusions extremal and Barabanov norms play an
interesting role. In this paper we restrict our attention to extremal
norms. Indeed, in the situations we consider we cannot guarantee that a
Barabanov norm exists.

\begin{definition}{}
    \label{d:exnorm}
    Let ${\cal S}= \bigcup_{t\geq0} {\cal S}_t \subset \R^{n \times n}$ be
    a semigroup. A norm $v$ on $\RR^n$ is called extremal for ${\cal S}$,
    if for all $x\in \RR^n$ and all $t\geq0$ we have
    \begin{equation}
        \label{eq:exnormineq}
        v(Sx) \leq \rho({\cal S})^t v(x) \,,\quad \forall S\in {\cal S}_t\,.
    \end{equation}
\end{definition}

In particular, this means that if the semigroup is exponentially stable,
which is equivalent to $\rho({\cal S})<1$, then an extremal norm is a
Lyapunov function that not only characterises exponential stability but
also the precise growth rate of the system.

The interesting fact is that existence of a $K$-irreducible element in
the convex hull of some ${\cal S}_t$ guarantees the existence of an
extremal norm. Note that in general extremal norms need not exist.  It is
known that existence is equivalent to the boundedness of the semigroup $\{
\rho^{-t}({\cal S}) S \midset t\geq 0, S\in {\cal S}_t \}$, \cite{Kozy90},
but this is a criterion that is hard to check in general.  An additional
benefit of the norm constructed here is that it can be chosen to be
monotone.

\begin{theorem}{}
    \label{t:nounbounded}
    Let $K\subset \R^n$ be a proper cone.  Let ${\cal S}\subset \pi(K)$ be
    a $K$-irreducible semigroup. If $\rho({\cal S})=1$, then ${\cal S}$ is
    bounded.
\end{theorem}

\begin{proof}
    Let $B$ be a compact base of $K$. Let $t>0$ be such that $\conv {\cal
      S}_t$ contains a $K$-irreducible element, which we denote by
    $A$. Assume to the contrary that ${\cal S}$ is unbounded. Thus
    there exists a sequence $\{ S_k \}_{k\in\N} \subset {\cal S}$, a
    sequence $c_k \to \infty$ and vectors $x_k\in B$ such that
    \begin{equation}
        \label{eq:unboundedcond}
        S_k x_k \in c_k B\,.
    \end{equation}
    As $A$ is irreducible, we have using (IR2) that
    \begin{equation}
        \label{eq:irredB}
        (I+A)^{n-1} B \subset \intt K\,.
    \end{equation}
    As $B$ is compact, so is the continuous image $(I+A)^{n-1} B$ and by
    Lemma~\ref{lem:base} there exists a constant $\delta >0$ such
    that for all $x\in B$ we have
    \begin{equation}
        \label{eq:irredB2a}
        (I+A)^{n-1} B \gg_K \delta x\,.
    \end{equation}
    Combining \eqref{eq:unboundedcond} and \eqref{eq:irredB2a} we obtain
    that
    \begin{equation}
        \label{eq:irredconc}
        (I+A)^{n-1}S_k x_k \gg_K \delta c_k x_k\,.
    \end{equation}
    By Lemma~\ref{lem:specrad} we obtain $r( (I+A)^{n-1}S_k) > \delta
    c_k$. Now by the binomial theorem
    \begin{equation*}
        (I+A)^{n-1} = \sum_{k=0}^{n-1}
        \begin{pmatrix}
            n-1\\ k
        \end{pmatrix}
        A^k
    \end{equation*}
    and so by defining
    \begin{equation}
\label{eq:etadef}
        \eta_n := \left(\sum_{k=0}^{n-1}
        \begin{pmatrix}
            n-1\\ k
        \end{pmatrix}\right)^{-1}
    \end{equation}
    we see that $\eta_n(I+A)^{n-1}S_k$ is a convex combination of
    elements of the set 
    \begin{equation*}
        \overline{M} = \{ S_k, AS_k,\ldots, A^{n-1}S_k \} \,.
    \end{equation*}
    On the other hand $A$ is a convex combination of elements in ${\cal
      S}_t$ and inductively we see, that $A^p$ is a convex combination of
    elements in ${\cal S}_{tp}$. Therefore there exists a set $\tilde M_k
    \subset {\cal S}$ such that 
    \begin{equation*}
        \eta_n(I+A)^{n-1}S_k \in \conv \tilde M_k\,.
    \end{equation*}
    Also, for $k$ large enough $r(\eta_n(I+A)^{n-1}S_k) >
    \eta_n\delta c_k > 1$. Using \eqref{eq:convspec} this implies
    \begin{equation}
        \label{eq:rhoconc}
        \rho(\tilde{M}_k, \N) > 1\,,
    \end{equation}
    and by \eqref{eq:gsr} some product of the matrices in $\tilde{M}_k$
    has a spectral radius larger than $1$. Now ${\cal S}(\tilde{M}_k,\N)
    \subset {\cal S}$ and so an element in ${\cal S}$ has a spectral
    radius larger than $1$. Again using \eqref{eq:gsr} this contradicts $ 1 =
    \rho({\cal S})$.  This contradiction completes the proof.
\end{proof}

The previous result now allows us to define monotone extremal norms, or even
absolute, extremal norms in the case that $K$ is simplicial.

\begin{theorem}
    \label{p:metextr}
    Let $K$ be a proper cone in $\R^n$. If ${\cal S}\subset \pi(K)$ is a
    $K$-irreducible semigroup, then there exists a monotone extremal norm
    $v$ for ${\cal S}$. If $K$ is simplicial, then the norm
    $v$ may be chosen to be absolute.
\end{theorem}

\begin{proof}
    As ${\cal S}$ is irreducible there exist $t>0$ and a $K$-irreducible
    $A\in {\cal S}_t$. Now $r(A)$ is a simple eigenvalue of $A$,
    \cite{KrasLifs89}, and so necessarily $r(A) >0$. Using
    \eqref{eq:convspec} this implies $\rho({\cal S})>0$. Thus considering
    the semigroup $\{ \rho^{-t}({\cal S}) S \midset t\geq 0, S\in {\cal
      S}_t \}$, we may assume without loss of generality that $\rho({\cal
      S})=1$. Using Theorem~\ref{t:nounbounded} it follows that ${\cal S}$
    is bounded.

    An extremal norm may then be defined in the following way, see also
    \cite{Kozy90}.  Let $\|\cdot\|$ be a $K$-monotone norm on
    $\RR^n$. Then define $v: \R^n \to \R$ by setting, for $x\in \R^n$,
    \begin{equation}
        \label{eq:defextnorm}
        v(x) := \sup \{ \| S x\| \midset S\in {\cal S} \} \,.
    \end{equation}
    It is clear that $v$ is positively homogeneous and positive definite,
    as $I\in{\cal S}$, and well-defined by boundedness of ${\cal S}$. The
    triangle inequality for $v$ follows from
    \begin{align*}
        v(x+y) = \sup \{ \| S (x + y)\| \midset S\in
        {\cal S} \}  \leq \sup \{ \| S x\| +\|S y\| \midset S\in
        {\cal S} \} \leq v(x) + v(y) \,.
    \end{align*}
    Using the assumption $\rho({\cal S})=1$ and the semigroup property of
    ${\cal S}$ we obtain extremality of $v$ from
    \begin{equation}
        \label{eq:df}
        v(Sx) = \sup \{ \| TS x \| \midset T\in {\cal S} \} \leq v(x)\,. 
    \end{equation}
    Finally the monotonicity of $v$ is inherited from the monotonicity of
    $\|\cdot\|$ as follows. For $0\leq_K x\leq_K y$
    we have $0\leq_K Sx \leq_K Sy$ for $S\in{\cal S}$ and so $\|Sx\| \leq
    \|Sy\|$ for all  $S\in{\cal S}$. Hence
    \begin{equation*}
       v(x) = \sup \{ \| S x\| \midset S\in {\cal S} \} \leq \sup \{ \| S y\| \midset S\in {\cal S} \}=
    v(y) \,.
    \end{equation*}

    If $K$ is simplicial, then we choose an absolute norm $\|\cdot\|$ to
    perform a variant of the construction described above. In this case we
    define
    \begin{equation*}
        v(x) := \sup \{ \| S |x|_K\| \midset S\in {\cal S} \} \,.
    \end{equation*}
    By definition $v$ satisfies $v(x) = v(|x|_K)$ for all $x\in
    \R^n$. Also if $0\leq_K |x|_K \leq_K |y|_K$, then
    \begin{equation}
        \label{eq:monabs}
        v(x) = \sup \{ \| S |x|_K\| \midset S\in {\cal S} \}
        \leq \sup \{ \| S |y|_K\| \midset S\in {\cal S} \} = v(y)\,.
    \end{equation}
  The triangle inequality for $v$ then follows from
    \begin{align*}
        v(x+y) = v(|x+y|_K) \stackrel{\eqref{eq:monabs}}{\leq}
        v(|x|_K+|y|_K) = \sup \{ \| S (|x|_K + |y|_K)\| \midset S\in {\cal
          S} \} \\ \leq \sup \{ \| S |x|_K\| +\|S |y|_K)\| \midset S\in
        {\cal S} \} \leq v(|x|_K) + v(|y|_K) = v(x) + v(y)\,.
    \end{align*}
    Positive definiteness and homogeneity are again clear and the
    extremality property follows as in \eqref{eq:df}. This concludes the
    proof. ~\hfill
\end{proof}

Combining the characterisation of $K$-irreducibility from
Proposition~\ref{lem:irredchar} we immediately obtain the following corollary.

\begin{corollary}
    \label{c:irredextnorm}
Let $K\subset \R^n$ be a proper cone.
    \begin{enum}
      \item If ${\cal M}\subset \pi(K)$ is compact and if $\conv {\cal M}$
        contains a $K$-irreducible element, then there exists a
        $K$-monotone, extremal norm $v$ for ${\cal S}$ generated by
        \eqref{eq:lininpos}.
      \item If ${\cal M}\subset \exppi(K)$ is a compact and if for every
        nontrivial face $F$ of $K$ there exists an $A\in {\cal M}$ such that
        $A\, \mathrm{span}\ F \not \subset \mathrm{span}\ F$, then there
        exists a $K$-monotone, extremal norm $v$ for ${\cal S}$ generated by
        \eqref{eq:lininmetz}.
    \end{enum}   
\end{corollary}

\section{Regularity of the Joint Spectral Radius}
\label{sec:regularity}

For $K$-irreducible matrices $A\in \RR^{n \times n}_+$ it is well known
that the spectral radius $\rho(A)$ is a simple eigenvalue of $A$ and that
all eigenvalues $\lambda$ of $A$ of modulus equal to the spectral radius
are simple. It is a consequence of standard perturbation theory, that
under these conditions the spectral radius as a function of the entries of
a matrix is Lipschitz continuous on a neighbourhood of $A$. In this
section we show that by the previous results the same is true for positive
linear inclusions that are $K$-irreducible.  This result complements the
result of \cite{Wirt02}, where it was shown that the joint spectral radius
is Lipschitz continuous on the set of compact matrix sets that are {\em
  irreducible} in the sense of representation theory. In this context this
name is a bit misleading, because in the nomenclature of \cite{Wirt02} a
set of matrices is irreducible, if no subspace other than the trivial
ones, $\{ 0 \} $ and $\RR^n$, is invariant under all matrices in ${\cal
  M}$. This property is not implied by the assumptions in
Theorem~\ref{p:metextr}.  To see this, consider a pair of positive
matrices with respect to $K = \R^{n \times n}_+$ with a common eigenvector
(take a set of row stochastic matrices for example).  Such a set will
automatically satisfy our assumptions but will clearly have a common
invariant subspace spanned by the common eigenvector.  Hence the set will
not be irreducible in the sense used in the work of Barabanov and others.
On the other hand, if a set of nonnegative or Metzler matrices has no
nontrivial common invariant subspace, it will be $K$-irreducible in our
sense.  Hence our assumption is strictly weaker than the usual one.

Let $A\in \R^{n \times n}$ and ${\cal M}\subset \R^{n \times n}$ be
closed. Then we define the distance from $A$ to ${\cal M}$ by
\begin{equation*}
    \dist(A,{\cal      M})\} := \min \{ \| A- M\| \midset M \in {\cal M} \}\,.
\end{equation*}
Note that as ${\cal M}$ is closed there exists a $B\in {\cal M}$ such that
$\|A-B\|= \dist(A,{\cal M})\}$.  The Hausdorff distance between compact
sets of matrices ${\cal M},{\cal N}$ is then defined by
\[H({\cal M},{\cal N}):= \max \{ \max_{A\in{\cal M}}\{ \dist(A,{\cal
  N})\},\max_{B\in{\cal N}} \{ \dist(B,{\cal M})\} \}\,.\] The particular
value of the Hausdorff distance depends on the norm we have chosen on
$\R^{n \times n}$, if we want to emphasise this we write
$H_{\|\cdot\|}({\cal M},{\cal N})$.

\begin{theorem}{}
    \label{t:lipschitzprop} 
Let $K\subset \R^n$ be a proper cone. 

{ (i) In the discrete time case
        the} joint spectral radius is locally Lipschitz continuous on the
      set
    \begin{equation*}
        {\cal P}_\N:= \{ {\cal M} \subset \pi(K) \midset {\cal M}\text{ is compact and $K$-irreducible}\} 
    \end{equation*}
    endowed with the Hausdorff metric.\\
    { (ii) In the continuous time case the joint spectral
    radius is locally Lipschitz continuous on the set
    \begin{equation*}
        {\cal P}_{\R_+}:= \{ {\cal M} \subset \exppi(K) \midset {\cal M}\text{ is compact and $K$-irreducible} \} 
    \end{equation*}
    endowed with the Hausdorff metric.}
\end{theorem}

In the proof we follow the idea of \cite{Wirt02}. There the proof is based
on the consideration of the eccentricity of extremal norms corresponding
to different sets. In general, the eccentricity of a norm $v$ with respect
to a norm $\|\cdot\|$ is defined by
\begin{equation}
    \label{eq:eccdef}
    \ecc_{\|\cdot\|}(v) := 
\frac{\max \{v(x)\;|\;\Vert x\Vert =1\}}{\min \{v(x)\;|\;\Vert x\Vert =1\}}\,.
\end{equation}
Note that for any $A\in \mathbb{R}^{n\times n}$ we have for the induced
operator norm that 
\begin{equation}
\label{27.06.02b}
\frac{1}{\ecc_{\|\cdot\|}(v)}\Vert A\Vert \leq v(A)\leq \ecc_{\|\cdot\|}(v)%
\Vert A\Vert \,.
\end{equation}

The decisive property is now that the eccentricity of absolute extremal
norms is bounded on compact subsets of ${\cal P}_{\N}$. We note that an
analogous statement to \cite[Lemma~4.1]{wirth2005structure} is false here,
because we cannot exclude the possibility of positively homogeneous
functions that have an extremality property and vanish on a subspace. 
An example to this effect is given by the pair of matrices
\begin{equation*}
    \begin{bmatrix}
        0&1\\1&0
    \end{bmatrix} \,,\quad
    \begin{bmatrix}
        1&0\\0&1
    \end{bmatrix}
\end{equation*}
one of which is clearly irreducible and for which the following function
is extremal, but of course not a norm:
\begin{equation*}
    w(x) := \left\|
      \begin{bmatrix}
          1 & 1 
      \end{bmatrix} x \right\|_2\,.
\end{equation*}

\begin{proposition}
    \label{p:boundedecc}
    Let $K\subset \R^n$ be a proper cone.  Let ${\cal X}\subset{\cal
      P}_\N$ be compact (as a subset of the metric space $({\cal P}_\N,
    H)$) and let $\|\cdot\|$ be $K$-monotone, then there exists a bound
    $0<C<\infty$ such that for all sets ${\cal M}\in {\cal X}$ there
    exists a $K$-monotone norm $v$, which is extremal for ${\cal S}({\cal
      M},\N)$ and satisfies
    \begin{equation}
        \label{eq:eccbound}
        \ecc_{\|\cdot\|}(v) < C\,.
    \end{equation}
\end{proposition}
\begin{proof}
    We show the property locally in a neighbourhood of ${\cal M}\in{\cal
      P}_\N$, then the assertion follows by a standard compactness argument.

    So let ${\cal M}\in {\cal P}_\N$ and apply Theorem~\ref{p:metextr} to
    choose a $K$-monotone, extremal norm for ${\cal M}$. We claim that
    there is a neighbourhood of ${\cal M}$ in ${\cal P}_\N$ for which
    \eqref{eq:eccbound} holds. If this is false then we may pick sequences
    ${\cal M}_k\to{\cal M}$ and $C_k\to\infty$ such that every
    $K$-monotone, extremal norm of ${\cal M}_k$ has eccentricity exceeding
    $C_k$.

    As norms are convex functions, a norm is extremal for ${\cal M}$ if
    and only if it is extremal for $\conv {\cal M}$. Thus we may assume
    that all ${\cal M}_k$ and ${\cal M}$ are convex. In particular, by
    assumption there is a $K$-irreducible matrix $\overline{M} \in {\cal M}$.
    By Remark~\ref{rem:irred}, for $k$ large enough there are
    $K$-irreducible matrices $\overline{M}_k \in {\cal M}_k$ with $\overline{M}_k
    \to \overline{M}$.

    As the joint spectral radius is continuous, \cite{Wirt02}, we know
    that $\rho({\cal M}_k) \to \rho({\cal M})>0$.  This shows that
    $\rho^{-1}({\cal M}_k) {\cal M}_k \to \rho({\cal M})^{-1}{\cal M}$. As
    this rescaling does not change extremal norms, we may assume that all
    joint spectral radii involved are equal to $1$.

    Let $\|\cdot\|$ be a $K$-monotone, extremal norm for ${\cal M}$. We
    construct $K$-monotone extremal norms $v_k$ for ${\cal M}_k$ using
    \eqref{eq:defextnorm}. This is possible by the construction in
    Theorem~\ref{p:metextr}. Note in particular, that this implies $v_k(x)
    \geq \|x\|$ for all $x\in \RR^n$.

    By (IR2) $\overline{M}$ is $K$-irreducible if and only if
    $(I+\overline{M})^{n-1}$ is $K$-positive.  Let $B$ be a compact 
    base of $K$.  As $\overline{M}_k \to \overline{M}$
    there exists by Lemma~\ref{lem:base} a constant $\delta>0$ and an
    index $k_0$, such that for all $k\geq k_0$ and all $x\in B$ we have
    \begin{equation}
        \label{eq:irredB2}
        (I+\overline{M}_k)^{n-1} B \gg_K \delta x\,.
    \end{equation}

    If $\ecc_{\|\cdot\|}(v_k)>C_k\to \infty$, then as $v_k(\cdot) \geq
    \|\cdot\|$ it follows { from the definition of $v_k$} that for all
    $k$ sufficiently large there are $S_k\in {\cal S}({\cal M}_k)$ and
    $x_k\in B$ such that (with $\eta_n$ defined by \eqref{eq:etadef})
    \begin{equation}
     \label{contra2}
        S_k x_k \in \frac{1}{\eta_n\delta }\,B\,,
    \end{equation}
    and so
    \begin{equation}
        \label{contra1}
        \eta_n (I+M_k)^{n-1} S_k x_k \gg_K x_k\,.
    \end{equation}
    As in the final step of the proof of Theorem~\ref{t:nounbounded} the
    combination of \eqref{contra1} and \eqref{contra2} leads to a
    contradiction to the assumption that $\rho({\cal M}_k)=1$. This
    contradiction to the assumption of unbounded eccentricity 
concludes
    the proof.
\end{proof}

Now the proof of Theorem~\ref{t:lipschitzprop} can be completed following
the steps outlined in \cite{Wirt02,wirth2005structure}.

\begin{proof}(of Theorem~\ref{t:lipschitzprop}) { (i)~In the discrete-time
      case,} let ${\cal X}\subset{\cal P}$ be compact, let $C$ be as in
    Proposition~\ref{p:boundedecc}. Pick ${\cal M},{\cal N}\in {\cal X}$
    and an absolute extremal norm $v$ for ${\cal M}$. Recall that by
    definition this implies for the induced matrix norm, also denoted by
    $v$, that $v(A)\leq \rho({\cal M})$ for all $A\in{\cal M}$.  Then for
    any $B\in {\cal N}$, we may choose $A\in {\cal M}$ such that
    $v(A-B)\leq \dist_v(B,{\cal M})$ and we obtain
    \begin{equation*}
        v(B) \leq v(A) + v(B-A) \leq \rho({\cal M}) + H_v({\cal M},{\cal N})\,,
    \end{equation*}
    where $H_v$ is the Hausdorff distance defined using $v$. This yields
    $\rho({\cal N})\leq \rho({\cal M}) + H_v({\cal M},{\cal N})$.  Using
    \eqref{27.06.02b}, we see that $H_v({\cal M},{\cal N})\leq C H({\cal
      M},{\cal N})$ and by symmetry the assertion follows.

    {(ii)~ The continuous time follows as in \cite{Wirt02} by noting that
      the map ${\cal M}\to {\cal S}_t({\cal M})$ defines a Lipschitz
      continuous set-valued map. If ${\cal M}$ consists of exponentially
      $K$-nonnegative matrices, then ${\cal S}_t({\cal M})\subset \pi(K)$
      and the irreducibility property is preserved for almost all $t$. In
      this way the continuous-time case is a direct consequence of the
      discrete-time case.}
\end{proof}

Note that the result of Theorem~\ref{t:lipschitzprop} does not yield the
full force of the statement for single $K$-irreducible matrices. There we
may obtain Lipschitz continuity of the spectral radius on a neighbourhood
which may also include matrices not in $\pi(K)$. So far our result is
restricted to neighbourhoods of $K$-nonnegative matrix sets, but we expect
it can be extended to larger neighbourhoods of irreducible sets of
nonnegative matrices. We note however that $K$-positivity is an open
property. We thus obtain immediately

\begin{corollary}
    \label{c:pos}
Let $K\subset \R^n$ be a proper cone.  

(i) The (discrete-time) joint spectral radius is locally Lipschitz continuous
on the set of compact subsets of 
\begin{equation*}
  \{ A \in \pi(K) \midset \text{ A is $K$-positive } \}\,,
\end{equation*}
endowed with the Hausdorff metric.\\
(ii) The (continuous-time) joint spectral radius is locally Lipschitz
continuous on the set of compact subsets of
\begin{equation*}
  \{ A \in \exppi(K) \midset \text{ A is exponentially $K$-positive } \}\,,
\end{equation*}
 endowed with the Hausdorff metric.
\end{corollary}

\section{Conclusions}

In this paper we have considered linear inclusions defining positive
systems. We show that under a generalised irreducibility assumption
absolute extremal norms exist. As an application local Lipschitz
continuity of the joint spectral radius on certain positive linear
inclusions is proved. The characterisation of irreducibility in terms of the
data of a semigroup of continuous-time systems with jumps similar to the
results of Propositions~\ref{p:irred} and \ref{lem:irredchar} remains an open
question.


\end{document}